\documentclass{article}

\usepackage{arxiv}

\usepackage[utf8]{inputenc} 
\usepackage[T1]{fontenc}    
\usepackage{hyperref}       
\usepackage{url}            
\usepackage{booktabs}       
\usepackage{amsfonts}       
\usepackage{nicefrac}       
\usepackage{microtype}      
\usepackage{lipsum}		
\usepackage{graphicx}
\usepackage{natbib}
\usepackage{doi}
\usepackage{amsmath, amsthm, amssymb, algorithm, algpseudocode}
\usepackage[table]{xcolor}
\usepackage{tikz}
\newcommand{\bra}[1]{\left(#1\right)}
\newcommand{\sbra}[1]{\left[#1\right]}
\newcommand{\seq}[1]{\left<#1\right>}
\newcommand{\seqA}[1]{{\left<#1\right>}_{A}}
\newcommand{\vertiii}[1]{{\left\vert\kern-0.25ex\left\vert\kern-0.25ex\left\vert #1
    \right\vert\kern-0.25ex\right\vert\kern-0.25ex\right\vert}}
\newcommand{\norm}[1]{\left\Vert#1\right\Vert}
\newcommand{\normA}[1]{{\left\Vert#1\right\Vert}_{A}}
\newcommand{\abs}[1]{\left\vert#1\right\vert}
\newcommand{\set}[1]{\left\{#1\right\}}
\renewcommand{\c}{\mathbb C}

\newcommand{\N}{\mathbb N}
\renewcommand{\r}{\mathrm{ran}}
\newcommand {\bh}{\mathcal{B}(\mathcal{H})}
\newcommand {\bah}{\mathcal{B}^{A}(\mathcal{H})}

\newcommand {\n}{\ker}

\newcommand {\Sa}{S^{\sharp_{A}}}

\newcommand {\h}{\mathcal{H}}
\newcommand {\x}{\mathbf{x}}
\newcommand {\Ta}{T^{\sharp_{A}}}

\renewcommand {\b}{\mathcal{B}}
\newcommand {\Xa}{X^{\sharp_{A}}}
\newcommand {\Ya}{Y^{\sharp_{A}}}

\newtheorem{theorem}{Theorem}[section]
\newtheorem{lemma}[theorem]{Lemma}
\newtheorem{proposition}[theorem]{Proposition}
\newtheorem{corollary}[theorem]{Corollary}
\newtheorem{definition}[theorem]{Definition}
\newtheorem{example}[theorem]{Example}

\newtheorem{remark}[theorem]{Remark}

\newcommand\mystyle{\everymath{\displaystyle}}
\mystyle
\title{Inequalities for the $A$-Norm and $A$-Numerical Radius of Operator Sums in Semi-Hilbertian Spaces with Applications}

\author{\href{https://orcid.org/0000-0002-3816-5287}{\includegraphics[scale=0.06]{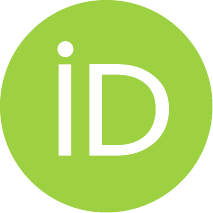}\hspace{1mm}M.H.M.~Rashid}\thanks{Corresponding Author} \\
	Department of Mathematics\&Statistics\\Faculty of Science P.O.Box(7)\\
	Mutah University University\\
	Mutah-Jordan \\
	\texttt{mrash@mutah.edu.jo}
}

\renewcommand{\shorttitle}{\textit{arXiv} Template}

\hypersetup{
pdftitle={Inequalities for the $A$-Norm and $A$-Numerical Radius of Operator Sums in Semi-Hilbertian Spaces with Applications},
pdfsubject={q-bio.NC, q-bio.QM},
pdfauthor={M.H.M.Rashid},
pdfkeywords={Positive operator; Semi-inner product; A-adjoint operator; A-numerical radius;   Inequality},
}

\begin{document}
\maketitle

\begin{abstract}
	 This paper establishes several new inequalities for the $A$-norm and $A$-numerical radius of operator sums in semi-Hilbertian spaces, significantly advancing the existing theory. We present two fundamental refinements of the generalized triangle inequality for operator norms, providing sharper estimates than previously known results. Our investigation yields novel bounds for the $A$-numerical radius of products and commutators of operators, with particular attention to their Cartesian decompositions. The developed framework enables applications to quantum mechanics, where we derive improved uncertainty relations and perturbation bounds, and to partial differential equations, where we obtain stability estimates for nonlocal elliptic operators. Through concrete examples, we demonstrate the optimality of our inequalities and their advantages over classical results. The theoretical contributions are complemented by potential applications in functional analysis, operator theory, and mathematical physics, suggesting directions for future research in semi-Hilbertian operator theory.
\end{abstract}

\keywords{Positive operator\and Semi-inner product\and A-adjoint operator\and A-numerical radius\and   Inequality}

\section{Introduction}
On a complex Hilbert space $\bra{\h,\seq{\cdot,\cdot}}$, let $\bh$ be the $C^*$-algebra containing all bounded linear operators, and let $\norm{\cdot}$ be its associated norm. The identity operator on $\h$ is denoted by the symbol $I$. $\r(T)$ represents the range of $T$ for any $T\in\bh$, and $\overline{\r(T)}$ represents the norm closure of $\r(T)$. In this study, we assume that $A\in\bh$ is a positive operator and that $P$, represented as $\overline{\r(A)}$, is the orthogonal projection onto the closure of the range of $A$. For every $x\in\h$, a positive operator $A$, represented as $A\geq 0$, satisfies $\seq{Ax,x}\geq 0$. A positive semidefinite sesquilinear form $\seq{\cdot,\cdot}_A:\h\times \h\to\c$, given as $\seq{x,y}_A=\seq{Ax,y}$ for $x,y\in\h$, is induced by such an operator $A$. $\norm{x}_A=\sqrt{\seq{x,x}_A}$ yields the seminorm $\normA{\cdot}$ generated by $\seq{\cdot,\cdot}_A$ for each $x\in\h$.  Notably, $\norm{\cdot}_A$ functions as a norm if and only if $A$ is an injective operator, and the space $\bra{\h,\norm{\cdot}_A}$ is complete if and only if $\r(A)$ is closed in $\h$. The semi-inner product $\seq{\cdot,\cdot}_A$ induces a semi-norm on a specific subspace of $\bh$. Specifically, for $T\in\bh$, if there exists $c>0$ such that $\norm{Tx}_A\leq c\norm{x}_A$ for all $x\in\overline{\r(A)}$, then
\begin{equation*}
 \normA{T}:=\sup_{\substack{x\in\overline{\r(A)}\\x\neq 0}}\frac{\normA{Tx}}{\normA{x}}=\inf\set{c>0:\normA{Tx}\leq c\normA{x},x\in \overline{\r(A)}}<\infty.
\end{equation*}

We set $\b^{A}(\h):=\set{T\in\bh:\normA{T}<\infty}$ . It can be seen that $\bah$ is not generally
a subalgebra of $\bh$ and $\normA{T}=0$ if and only if $ATA=0$. In addition, for $T\in\bah$,
we have
\begin{equation*}
  \normA{T}=\sup\set{\abs{\seqA{Tx,y}}:x,y\in\overline{\r(A)},\normA{x}=\normA{y}=1}.
\end{equation*}
An operator $T$ is called $A$-positive if $AT \geq 0$. Note that if $T$ is $A$-positive, then
\begin{equation*}
  \norm{T}_A=\sup\set{\seqA{Tx,x}:x\in\h,\norm{x}_A=1}.
\end{equation*}

If, for every $x,y\in\h$, the equality $\seqA{Tx,y}=\seqA{x,Sy}$ holds, that is, $AS=T^*A$, then an operator $S\in\bh$ is a $A$-adjoint of $T$. It is crucial to remember that there is no guarantee that a $A$-adjoint operator will exist. An operator $T\in\bh$ might really have one, several, or none of the $A$-adjoints. $\b_A(\h)$ represents the set of all operators that have $A$-adjoints. Notably, $\b_A(\h)$ creates a subalgebra of $\bh$ that is neither dense in $\bh$ nor closed. Additionally, equality is attained if $A$ is injective and has a closed range, and the inclusions $\b_A(\h)\subseteq \bah\subseteq \bh$ hold. This is further illustrated by the use of Douglas' Theorem \cite{Douglas}.

\begin{eqnarray}\label{Ineq.A1}
  \b_A(\h) &=& \set{T\in\bh:\r(T^*A)\subseteq \r(A)} \\
   &=& \set{T\in\bh:\exists \lambda>0: \norm{ATx}\leq \lambda\norm{Ax},\,\forall x\in\h}. \nonumber
\end{eqnarray}

If $T\in\b_A(\h)$, the reduced solution of the equation $AX=T^*A$  is a distinguished
$A$-adjoint operator of $T$, which is denoted by $T^{\sharp_A}$; see \cite{MKX}. Note that $T^{\sharp_A}=A^{\dag}T^*A$
in which $A^{\dag}$ is the Moore-Penrose inverse of $A$. It is useful that if $T\in\b_A(\h)$, then
$AT^{\sharp_A}=T^*A$. An operator $T\in\bh$ is said to be $A$-selfadjoint if $AT$ is selfadjoint,
i.e., $AT=T^*A$. Observe that if $T$ is $A$-selfadjoint, then $T\in\b_A(\h)$. However, it does
not hold, in general, that $T=\Ta$. For example, consider the operators $A=\begin{bmatrix}1& 1 \\1 &1 \\\end{bmatrix}$
and $T=\begin{bmatrix}2& 2 \\0 &0 \\\end{bmatrix}$. Then simple computations show that $T$ is $A$-selfadjoint and
$\Ta=\begin{bmatrix}1& 1 \\1 &1 \\\end{bmatrix}\neq T$. More precisely, if $T\in\b_A(\h)$, then $\Ta\in\b_A(\h)$ if and only if $T$ is $A$-selfadjoint
and $\overline{\r(A)}\subseteq \r(A)$. Notice that if $T\in\b_A(\h)$, then $\Ta\in\b_A(\h)$, $(\Ta)^{\sharp_A}=PTP$ and
$\bra{(\Ta)^{\sharp_A}}^{\sharp_A}=\Ta$. In addition, $\Ta T$ and $T\Ta$ are $A$-selfadjoint and $A$-positive, and so we
have
\begin{equation}\label{Eq.1.1}
  \normA{\Ta T}=\normA{T\Ta}=\normA{T}^2=\normA{\Ta}^2.
\end{equation}
Furthermore, if $T, S\in\b_A(\h)$, then $(TS)^{\sharp_A}=S^{\sharp_A}\Ta$, $\normA{TS}\leq \normA{T}\normA{S}$ and
$\normA{Tx}\leq \normA{T}\normA{x}$ for all $x\in\h$.

For proofs and more facts about this class of operators, we refer the reader to \cite{ACG1, ACG2}
and their references.

Again, by applying Douglas' theorem, it can observed that
\begin{equation*}
  \b_{A^{1/2}}(\h)=\set{T\in\bh:\exists \lambda>0: \normA{Tx}\leq \lambda\normA{x},\,\forall x\in\h}.
\end{equation*}
Operators in $\b_{A^{1/2}}(\h)$ are called A-bounded. Notice that, if $T\in \b_{A^{1/2}}(\h)$, then
$T(\n(T))\subseteq \n(T)$. Moreover, it was proved in \cite{FagGor} that if  $T\in \b_{A^{1/2}}(\h)$, then
\begin{eqnarray}\label{Ineq.A2}
  \normA{T} &=& \sup\set{\normA{Tx}:x\in\h,\normA{x}=1} \\
   &=& \sup\set{\abs{\seqA{Tx,y}}:\normA{x}=\normA{y}=1,x,y\in\h}. \nonumber
\end{eqnarray}

Note that the subspaces $\b_A(\h)$ and $\b_{A^{1/2}}(\h)$ are two subalgebras of $\bh$,  which
are neither closed nor dense in $\bh$. Moreover, we have the following inclusions
\begin{equation}\label{Ineq.A3}
  \b_A(\h) \subseteq \b_{A^{1/2}}(\h)\subseteq \bah\subseteq \bh.
\end{equation}
For an operator $T\in\b_A(\h)$, we write
\begin{equation*}
  Re_A(T)=\frac{1}{2}\bra{T+\Ta}\quad\mbox{and}\quad Im_A(T)=\frac{1}{2i}\bra{T-\Ta}.
\end{equation*}
For $T\in\b_A(\h)$, $A$-numerical radius of $T$, denoted as $w_A(T)$ (see \cite{BFS}), is defined as
\begin{equation}\label{Num-A}
  w_A(T)=\sup\set{\abs{\seqA{Tx,x}}:x\in\h,\normA{x}=1}.
\end{equation}
Also, for $T\in\b_A(\h)$, the $A$-Crawford number of $T$, denoted as $c_A(T)$ (see \cite{Zamani-1}), is
defined as
\begin{equation}\label{Craf}
  c_A(T)=\inf\set{\abs{\seqA{Tx,x}}:x\in\h,\normA{x}=1}.
\end{equation}
For $T\in\b_A(T)$, it is well known that the $A$-numerical radius of $T$ is equivalent to
$A$-operator semi-norm of $T$, (see \cite{Zamani-2}), satisfying the following inequality:
\begin{equation}\label{equiv-Num-norm}
  \frac{1}{2}\normA{T}\leq w_A(T)\leq \normA{T}.
\end{equation}
 One of the
important properties of $w_A(\cdot)$ is that it is weakly $A$-unitarily invariant (see \cite{BPN}),
that is,
\begin{equation}\label{A-unitary}
  w_A\bra{U^{\sharp_A}TU}=w_A(T),
\end{equation}
for every $A$-unitary $U\in\b_A(\h)$. Another basic fact about the $A$-numerical radius
is the power inequality (see \cite{MXZ}), which asserts that
\begin{equation}\label{Power}
  w_A(T^n)\leq w_A^n(T)\,\,(n\in\N).
\end{equation}
It is evident from the first inequality in (\ref{equiv-Num-norm}) that if $T,S\in\b_A(\h)$, then
\begin{equation}\label{product-1}
  w_A(TS)\leq 4w_A(T)w_A(S).
\end{equation}
If $TS=ST$, then
\begin{equation}\label{product-2}
  w_A(TS)\leq 2w_A(T)w_A(S).
\end{equation}
Moreover, if $T,S,\in\b_A(\h)$ are $A$-normal, then
\begin{equation}\label{product-3}
  w_A(TS)\leq w_A(T)w_A(S).
\end{equation}

The study of numerical radius inequalities has a rich history in operator theory, with foundational contributions documented in \cite{ACG1, ACG2, ACG3, BFP, BPN, FagGor, Feki, Zamani-1, Zamani-2}. These investigations have been particularly fruitful in the context of semi-Hilbertian spaces, where Zamani's work \cite{Zamani-1, Zamani-2} established fundamental inequalities for the $A$-numerical radius that have become cornerstones of the theory. Building on these foundations, Ahmad \cite{Ahmad} introduced an innovative norm for $A$-bounded linear operators that generalizes the $A$-numerical radius norm, while recent comprehensive treatments in \cite{Dolat, KitRash} have further advanced the field through systematic analyses of operator matrix inequalities.

Our present work makes significant contributions to this active research area by developing new inequalities for the $A$-norm and $A$-numerical radius of operator sums in Hilbert spaces. Specifically, we refine the generalized triangle inequality for operator norms (Theorems~\ref{TheoremA1} and~\ref{TheoremB1}), improve bounds for operators with Cartesian decompositions (see the remark following Theorem~\ref{TheoremB1}), and establish novel relationships between the $A$-numerical radius and operator products (Theorem~\ref{TheoremQA1} and Corollary~\ref{Corollary3.7}). These results not only extend the existing literature but also provide sharper tools for analyzing operator behavior in semi-Hilbertian settings. The applications we develop, particularly those involving Cartesian decompositions, offer new perspectives on classical problems while demonstrating the practical utility of our theoretical advances.

\section {$A$-norm inequalities}
Several inequalities are provided in this section for the operator $A$-norm of sums of bounded linear operators in semi-Hilbertian spaces.
\begin{theorem}\label{TheoremA1} For any $X_1,\cdots, X_n\in\b_A(\h)$ we have
\begin{equation}\label{RA1}
  \normA{\sum_{k=1}^{n} X_k}^2\leq \normA{\sum_{k=1}^{n}X_k^{\sharp_A}X_k}^2+\frac{1}{2}\normA{(n-2)\sum_{k=1}^{n}
  X_k^{\sharp_A}X_k+\sum_{k=1}^{n}
  X_k^{\sharp_A}\sum_{k=1}^{n}X_k}.
\end{equation}
\end{theorem}
\begin{proof} Utilising the elementary inequality
\begin{equation}\label{RA2}
  Re\seqA{a,b}\leq \frac{1}{4}\normA{a+b}^2,\,\, a,b\in\h,
\end{equation}
we then have
\begin{eqnarray*}
  \sum_{\substack{1\leq k,j\leq n\\ j\neq k}}^{n}Re\seqA{X_k x,X_j x} &\leq&\frac{1}{4}\sum_{\substack{1\leq k,j\leq n\\ j\neq k}}^{n}
  \normA{X_k x+X_j x}^2 \\
   &=&\frac{1}{4}\sum_{\substack{1\leq k,j\leq n\\ j\neq k}}^{n}\seqA{\bra{X_k+X_j}^{\sharp_A}\bra{X_k+X_j}x,x}\\
   &=& \frac{1}{4}\sum_{\substack{1\leq k,j\leq n\\ j\neq k}}^{n}\seqA{\bra{X_k^{\sharp_A}+X_j^{\sharp_A}}\bra{X_k+X_j}x,x}
\end{eqnarray*}
Now, on making use of the following identity
\begin{eqnarray}\label{RA3}
  \normA{\sum_{k=1}^{n} X_k}^2&=&Re\sbra{\sum_{j,k=1}^{n}\seqA{X_k x,X_j x}}=\sum_{j,k=1}^{n}Re\seqA{X_k x,X_j x}\nonumber \\
   &=& \sum_{k=1}^{n}\normA{X_k x}^2+\sum_{\substack{1\leq k,j\leq n\\ j\neq k}}^{n}Re\seqA{X_k x,X_j x}\\
   &=&\seqA{\bra{\sum_{k=1}^{n}X_k^{\sharp_A}X_k}x,x}+Re\seqA{\bra{\sum_{\substack{1\leq k,j\leq n\\ j\neq k}}^{n}X_k^{\sharp_A}X_k}x,x},\nonumber
\end{eqnarray}
we can state
\begin{equation}\label{RA4}
  \normA{\sum_{k=1}^{n} X_k}^2\leq \seqA{\bra{\sum_{k=1}^{n}X_k^{\sharp_A}X_k}x,x}+
  \frac{1}{4}\sum_{\substack{1\leq k,j\leq n\\ j\neq k}}^{n}\seqA{\bra{X_k^{\sharp_A}+X_j^{\sharp_A}}\bra{X_k+X_j}x,x}
\end{equation}
for any $x\in\h$. Since
\begin{eqnarray}\label{RA5}
  &&\sum_{\substack{1\leq k,j\leq n\\ j\neq k}}^{n}\seqA{\bra{X_k^{\sharp_A}+X_j^{\sharp_A}}\bra{X_k+X_j}x,x}\nonumber\\
  &&=\sum_{k,j=1}^{n}\bra{X_k^{\sharp_A}X_k+X_k^{\sharp_A}X_j+X_j^{\sharp_A}X_k+X_j^{\sharp_A}X_j}
  -4\sum_{k=1}^{n}X_k^{\sharp_A}X_k\nonumber\\
  &&=2n\sum_{k=1}^{n}X_k^{\sharp_A}X_k+2\sum_{k=1}^{n}X_k^{\sharp_A}\sum_{k=1}^{n}X_k-4\sum_{k=1}^{n}X_k^{\sharp_A}X_k\nonumber\\
  &&=2\sbra{(n-2)\sum_{k=1}^{n}X_k^{\sharp_A}X_k+\sum_{k=1}^{n}X_k^{\sharp_A}\sum_{k=1}^{n}X_k}.
\end{eqnarray}
Combining the inequalities (\ref{RA4}) and (\ref{RA5}), we obtain
\begin{equation*}
  \normA{\sum_{k=1}^{n} X_k}^2\leq \seqA{\bra{\sum_{k=1}^{n}X_k^{\sharp_A}X_k}x,x}+\frac{1}{2}\sbra{(n-2)\sum_{k=1}^{n}X_k^{\sharp_A}X_k+\sum_{k=1}^{n}X_k^{\sharp_A}\sum_{k=1}^{n}X_k}.
\end{equation*}
Taking the supremum over all vectors $x\in \h$ with $\normA{x}=1$ in the above inequality, we get the desired inequality.
\end{proof}
\begin{remark} Since, by the triangle inequality we have that
$$\normA{(n-2)\sum_{k=1}^{n}X_k^{\sharp_A}X_k+\sum_{k=1}^{n}X_k^{\sharp_A}\sum_{k=1}^{n}X_k}
\leq (n-2)\normA{\sum_{k=1}^{n}X_k^{\sharp_A}X_k}+\normA{\sum_{k=1}^{n}X_k}^2.$$
Hence by the inequality (\ref{RA1}) we deduce that
$$\normA{\sum_{k=1}^{n} X_k}^2\leq n  \normA{\sum_{k=1}^{n}X_k^{\sharp_A}X_k}.$$
\end{remark}
\begin{remark} The case $n=2$ provides the following interesting inequality
\begin{equation}\label{RA6}
  \normA{\frac{B+C}{2}}^2\leq \normA{\frac{B^{\sharp_A}B+C^{\sharp_A}C}{2}}
\end{equation}
for any $B,C\in\b_A(\h)$. If in this inequality we choose $B=X$ and $C=\Xa$, then we get
\begin{equation}\label{RA7}
  \normA{\frac{X+\Xa}{2}}^2\leq \normA{\frac{\Xa X+X\Xa}{2}}
\end{equation}
for any $X\in\b_A(\h)$.

Moreover, if $T=X+iY$ is the Cartesian decomposition of $T$ with $X=\frac{T+\Ta}{2}$ and $Y=\frac{T-\Ta}{2i}$, then we get
 \begin{equation}\label{RA8}
   \normA{T}^2\leq \normA{\Ta T+T\Ta}
 \end{equation}
 for any $T\in\b_A(\h)$.
\end{remark}
\begin{theorem}\label{TheoremB1} For any $X_1,\cdots, X_n\in\b_A(\h)$ we have
\begin{eqnarray}\label{RB1}
  \normA{\sum_{k=1}^{n} X_k}^2&\leq& \normA{\sum_{k=1}^{n}X_k^{\sharp_A}X_k}+\frac{1}{2}
  \normA{\sum_{\substack{1\leq k,j\leq n\\ j\neq k}}^{n}X_k^{\sharp_A}X_k}^2+\frac{1}{2}\\
   &=& \normA{\sum_{k=1}^{n}X_k^{\sharp_A}X_k}+\frac{1}{2}\normA{\sum_{j=1}^{n}X_j^{\sharp_A}\sum_{k=1}^{n}X_k-
   \sum_{k=1}^{n}X_k^{\sharp_A}X_k}^2+\frac{1}{2}.\nonumber
\end{eqnarray}
\end{theorem}
\begin{proof} For any $x\in\h$, observe that
\begin{equation*}
  \normA{\sum_{k=1}^{n} X_k}^2=\seqA{\bra{\sum_{k=1}^{n}X_k^{\sharp_A}X_k}x,x}+Re\seqA{\bra{\sum_{\substack{1\leq k,j\leq n\\ j\neq k}}^{n}X_k^{\sharp_A}X_k}x,x}.
\end{equation*}
Then by using the inequality (\ref{RA2}), we have
\begin{equation*}
  \normA{\sum_{k=1}^{n} X_k}^2\leq \seqA{\bra{\sum_{k=1}^{n}X_k^{\sharp_A}X_k}x,x}+\frac{1}{2}
  \sbra{\normA{\sum_{\substack{1\leq k,j\leq n\\ j\neq k}}^{n}\bra{X_j^{\sharp_A}X_k}x}^{2}+\normA{x}^2}.
\end{equation*}
Taking the supremum over all vectors $x\in\h$ with $\normA{x}=1$, we get
\begin{equation}\label{RB2}
  \normA{\sum_{k=1}^{n} X_k}^2\leq \normA{\sum_{k=1}^{n}X_k^{\sharp_A}X_k}+
  \frac{1}{2}\normA{\sum_{\substack{1\leq k,j\leq n\\ j\neq k}}^{n}\bra{X_j^{\sharp_A}X_k}x}^{2}+\frac{1}{2}.
\end{equation}
Since
$$\sum_{\substack{1\leq k,j\leq n\\ j\neq k}}^{n}X_j^{\sharp_A}X_k=\sum_{j=1}^{n}X_j^{\sharp_A}\sum_{k=1}^{n}X_k-\sum_{k=1}^{n}
X_k^{\sharp_A}X_k.$$
The last part of the inequality (\ref{RB1}) is also proved.
\end{proof}
\begin{remark} For $n=2$, we can state that
\begin{equation}\label{RT1}
  \normA{T+S}^2\leq \normA{\Ta T+\Sa S}+\frac{1}{2}\normA{\Ta S+\Sa T}^2+\frac{1}{2}
\end{equation}
for any $T,S\in\b_A(\h)$. If in this inequality we choose $T=X$ and $S=\Xa$, then we get
\begin{equation}\label{RT2}
  \normA{X+\Xa}^2\leq \normA{\Xa X+X\Xa}+\frac{1}{2}\normA{\bra{\Xa}^{2} +X^2}^2+\frac{1}{2}
\end{equation}
for any $X\in\b_A(\h)$.

Now, if $T=X+iY$ is the Cartesian decomposition of $T$, then applying (\ref{RT1}) for $X$ and $Y$, we get
\begin{equation}\label{RT3}
  \normA{T}^2\leq \frac{1}{2}\normA{\Ta T+T\Ta}+\frac{1}{4}\normA{\Ta T-T\Ta}^2+\frac{1}{2}
\end{equation}
for any $T\in\b_A(\h)$.
\end{remark}
A similar approach which provides another inequality for the operator
$A$-norm is incorporated in:
\begin{theorem}For any $X_1,\cdots, X_n\in\b_A(\h)$ we have
  \begin{equation}\label{CT1}
    \normA{\sum_{k=1}^{n}X_k}^2+\sum_{k=1}^{n}\normA{X_k}^2
    \leq \normA{\sum_{k=1}^{n}X_k^{\sharp_A}X_k}+\frac{1}{4}\sum_{j,k=1}^{n}\normA{X_j+X_k}^2.
  \end{equation}
\end{theorem}
\begin{proof} Since, by the proof of Theorem \ref{TheoremA1}, we have
  $$\sum_{\substack{1\leq k,j\leq n\\ j\neq k}}^{n}Re\seqA{X_k x,X_j x}\leq \frac{1}{4}
  \sum_{\substack{1\leq k,j\leq n\\ j\neq k}}^{n}\normA{X_k+X_j}^2,$$
  then, by (2.2), we can state that:
  \begin{equation}\label{CT2}
    \normA{\sum_{k=1}^{n} X_k x}^2=\seqA{\bra{\sum_{k=1}^{n}X_k^{\sharp_A}X_k}x,x}+\frac{1}{4}
  \sum_{\substack{1\leq k,j\leq n\\ j\neq k}}^{n}\normA{(X_k+X_j)x}^2
  \end{equation}
  for every $x\in\h$ with $\normA{x}=1$.

  Taking the supremum over all vectors $x\in\h$, $\normA{x}=1$, we deduce that
 \begin{eqnarray*}
   \normA{\sum_{k=1}^{n} X_k x}^2 &\leq& \normA{\sum_{k=1}^{n}X_k^{\sharp_A}X_k}+\frac{1}{4} \sum_{\substack{1\leq k,j\leq n\\ j\neq k}}^{n}
   \normA{X_k+X_j}^2\\
    &=& \normA{\sum_{k=1}^{n}X_k^{\sharp_A}X_k}+\frac{1}{4}\sum_{j,k=1}^{n}\normA{X_k+X_j}^2-\sum_{k=1}^{n}\normA{X_k}^2
 \end{eqnarray*}
 which is exactly the desired result (\ref{CT1}).
\end{proof}
It follows by the identity (\ref{RA3}) on noticing that
$$Re\seqA{\sum_{\substack{1\leq k,j\leq n\\ j\neq k}}^{n}\bra{X_k^{\sharp_A}X_k}x,x}\leq \frac{1}{4}
\normA{\sum_{\substack{1\leq k,j\leq n\\ j\neq k}}^{n}\bra{X_k^{\sharp_A}X_k+I}x}^2$$
for every $x\in\h$, the following result.
\begin{theorem} If $X_1,\cdots, X_n\in\b_A(\h)$, then
\begin{equation}\label{TD1}
  \normA{\sum_{k=1}^{n}X_k}^2+\sum_{k=1}^{n}\normA{X_k}^2
    \leq \normA{\sum_{k=1}^{n}X_k^{\sharp_A}X_k}+\frac{1}{4}\normA{\sum_{\substack{1\leq k,j\leq n\\ j\neq k}}^{n}X_k^{\sharp_A}X_k+I}^2.
\end{equation}
\end{theorem}
\begin{proof}It follows by the identity (\ref{RA3}) on noticing that
$$Re\seqA{\sum_{\substack{1\leq k,j\leq n\\ j\neq k}}^{n}\bra{X_k^{\sharp_A}X_k}x,x}\leq \frac{1}{4}
\normA{\sum_{\substack{1\leq k,j\leq n\\ j\neq k}}^{n}\bra{X_k^{\sharp_A}X_k+I}x}^2$$
for every $x\in\h$. The details are omitted.
\end{proof}
\section{$A$-Numerical Radius Inequalities For
Several Operators}
This section's primary objective is to obtain a number of upper bounds for the $A$-numerical radius that are improvements on certain preexisting ones.
\begin{lemma}\label{LemmaTT1} Let $T,S\in\b_A(\h)$. Then
\begin{eqnarray}\label{TT1}
  \normA{T\Ta+S\Sa} &\leq& \max\set{\normA{T+S}^2,\normA{T-S}^2}\nonumber \\
   &-&\frac{\abs{\normA{T+S}^2-\normA{T-S}^2}}{2}.
\end{eqnarray}
In particular, letting $S=\Ta$, we have
\begin{eqnarray}\label{TT2}
  \normA{T\Ta+\Ta T}  &\leq&4\max\set{\normA{Re_{A} T}^2,\normA{Im_{A} T}^2}\nonumber \\
   &-&2\abs{\normA{Re_{A} T}^2-\normA{Im_{A} T}^2}.
\end{eqnarray}
\end{lemma}
\begin{proof} We have
\begin{eqnarray*}
  &&\max\set{\normA{T+S}^2,\normA{T-S}^2}=\max\set{\normA{\Ta+\Sa}^2,\normA{\Ta-\Sa}^2}\\
  &&=\frac{\normA{\Ta+\Sa}^2+\normA{\Ta-\Sa}^2}{2}+\frac{\normA{\Ta+\Sa}^2-\normA{\Ta-\Sa}^2}{2}\\
  &&= \frac{\normA{\abs{\Ta+\Sa}^2}+\normA{\abs{\Ta-\Sa}^2}}{2}+\frac{\abs{\normA{T+S}^2-\normA{T-S}^2}}{2}\\
  &&\geq \frac{\normA{\abs{\Ta+\Sa}^2+\abs{\Ta-\Sa}^2}}{2}+\frac{\abs{\normA{T+S}^2-\normA{T-S}^2}}{2}\\
  &&=\normA{T\Ta+S\Sa}+\frac{\abs{\normA{T+S}^2-\normA{T-S}^2}}{2}.
\end{eqnarray*}
\end{proof}
\begin{theorem}\label{TheoremQA1} For any $T,S\in\b_A(\h)$,
\begin{equation}\label{QA1}
  w_A(TS+ST)\leq 2\sqrt{2}\normA{T}w_A(S).
\end{equation}
\end{theorem}
\begin{proof} First note that $w_A(T)\leq 1$ and $\normA{x}=1$ we have
\begin{equation}\label{QA2}
  \normA{Tx}^2+\normA{\Ta x}^2\leq 4.
\end{equation}
The inequality (\ref{QA2})may be obtained by noting that
\begin{equation}\label{QA3}
  Re\seqA{e^{i\theta}Tg(\theta),g(\theta)}\leq \normA{g(\theta)}^2,
\end{equation}
where $g(\theta)=\frac{1}{2}e^{i\theta}Tx+\frac{1}{2}e^{-i\theta}\Ta x+x$, and integrating (\ref{QA3}) over $[0,2\pi]$ to obtain
$$\frac{1}{2}\normA{Tx}^2+\frac{1}{2}\normA{\Ta x}^2\leq \frac{1}{4}\normA{Tx}^2+\frac{1}{4}\normA{\Ta x}^2+\normA{x}^2. $$
From (\ref{QA2}) it follows that  $\normA{Tx}+\normA{\Ta x}\leq 2 \sqrt{2}$
so that when $\normA{S}\leq 1,w_A(T)\leq 1$ and $\normA{x}\leq 1$ we have
\begin{eqnarray*}
  \abs{\seqA{(TS+ST)x,x}} &\leq&\normA{Tx}\normA{\Sa}+\normA{\Ta x}\normA{S} \\
   &\leq&\normA{T x}+\normA{\Ta x}\leq 2\sqrt{2}.
\end{eqnarray*}
\end{proof}
The following example illustrates Theorem \ref{TheoremQA1}.
\begin{example} Let $T=\begin{bmatrix}1& 0 \\1 & 0 \\\end{bmatrix}$, $S=\begin{bmatrix}1& 1 \\0 & 0 \\\end{bmatrix}$
and $A=\begin{bmatrix}1& -1 \\-1 & 2 \\\end{bmatrix}$. Then elementary calculations show that
$w_A(TS+ST)=\frac{1+\sqrt{10}}{2}\approx 2.08$,$\normA{T}=1$ and $w_A(S)=1$. Hence
$$2.08\approx w_A(TS+ST)\leq 2\sqrt{2}\normA{T}w_A(S)\approx 2\sqrt{2}=2.828.$$
\end{example}
Our refinement of the inequality (\ref{QA2}) can be stated as follows.
\begin{lemma}\label{LemmaTT2} Let $T\in\b_A(\h)$ such that $w_A(T)\leq 1$, and let $x\in\h$ with $\normA{x}=1$. Then
\begin{equation}\label{QA5}
  \normA{Tx}^2+\normA{\Ta x}^2\leq 4\bra{1-\frac{\abs{\normA{Re_{A} T}^2-\normA{Im_{A} T}^2}}{2}}.
\end{equation}
\end{lemma}
\begin{proof} It follows from the inequality (\ref{TT2}) that
\begin{eqnarray*}
 \normA{Tx}^2+\normA{\Ta x}^2 &=& \abs{\seqA{\bra{T\Ta+\Ta T}x,x}}\leq \normA{T\Ta+\Ta T} \\
   &=&4\max\set{\normA{Re_{A} T}^2,\normA{Im_{A} T}^2}-2\abs{\normA{Re_{A} T}^2-\normA{Im_{A} T}^2}\,\bra{\mbox{by Lemma \ref{LemmaTT1}}}\\
   &=&4\max\set{w_A^2(Re_{A} T),w_A^2(Im_{A} T)}-2\abs{\normA{Re_{A} T}^2-\normA{Im_{A} T}^2}\\
   &\leq& 4w_A^2(T)-2\abs{\normA{Re_{A} T}^2-\normA{Im_{A} T}^2}\\
   &\leq&4\bra{1-\frac{\abs{\normA{Re_{A} T}^2-\normA{Im_{A} T}^2}}{2}}.
\end{eqnarray*}
\end{proof}

The following example illustrates Lemma \ref{LemmaTT2}.
\begin{example} Let $T=\frac{1}{2}\begin{bmatrix}1& 0 \\1 & 1 \\\end{bmatrix}$, $A=\begin{bmatrix}1& -1 \\-1 & 2 \\\end{bmatrix}$
and $x=\frac{1}{2}\begin{bmatrix} 2-\sqrt{3} \\ 1-\sqrt{3} \\\end{bmatrix}$. Then elementary calculations show that
$w_A(T)=1$, $\normA{x}=1$, $\normA{Tx}^2\approx 0.5559$, $\normA{\Ta x}^2\approx 1.0959$, $\normA{Re_{A} T}^2=1$, $\normA{Im_{A} T}^2=0.5$.

Consequently,
$$1.6518\approx \normA{Tx}^2+\normA{\Ta x}^2\leq 4\bra{1-\frac{\abs{\normA{Re_{A} T}^2-\normA{Im_{A} T}^2}}{2}}\approx 3.$$
\end{example}
Based on Lemma \ref{LemmaTT2}, we have the following general numerical radius inequality.
\begin{theorem}\label{TheoremGN1} Let $T,S X,Y\in\b_A(\h)$. Then
\begin{eqnarray*}
  w_A(TXS\pm SYT) &=&2\sqrt{2}\normA{S}\max\set{\normA{X},\normA{Y}}\sqrt{w_A^2(T)-\frac{\abs{\normA{Re_{A} T}^2-\normA{Im_{A} T}^2}}{2}}.
\end{eqnarray*}
\end{theorem}
\begin{proof} First, suppose that $w_A(T)\leq 1$,$\normA{X}\leq 1, \normA{Y}\leq 1$, and let $x\in\h$ with $\normA{x}=1$. Then
\begin{eqnarray*}
  \abs{\seqA{\bra{TX\pm YT}x,x}} &=&\abs{\seqA{Xx,\Ta x}+\seqA{Tx,\Ya x}} \\
   &\leq&\normA{X x}\normA{\Ta x}+\normA{Tx}\normA{\Ya x}\\
   &\leq & \normA{X}\normA{\Ta x}+\normA{Tx}\normA{\Ya}\\
   &\leq & \normA{\Ta x}+\normA{Tx}\\
   &\leq& \sqrt{2}\bra{\normA{\Ta x}^2+\normA{Tx}^2}^{1/2}\\
   &\leq&2\sqrt{2}\sqrt{1-\frac{\abs{\normA{Re_{A} T}^2-\normA{Im_{A} T}^2}}{2}}\,\,\bra{\mbox{by Lemma \ref{LemmaTT2}}},
\end{eqnarray*}
and so
\begin{eqnarray}\label{GN1}
  w_A(TX\pm YT) &=&\sup_{\normA{x}=1}\abs{\seqA{\bra{TX\pm YT}x,x}}\nonumber  \\
   &\leq&2\sqrt{2}\sqrt{1-\frac{\abs{\normA{Re_{A} T}^2-\normA{Im_{A} T}^2}}{2}}.
\end{eqnarray}
For the general case, let $T$, $X$, and $Y$ be any operators in $\b_A(\h)$. It is clear
that the result is trivial if $w_A(T)=0$ or $\max\set{\normA{X},\normA{Y}}=0$, so suppose that
$w_A(T)\neq 0$ and $\max\set{\normA{X},\normA{Y}}\neq 0$. In the inequality (\ref{GN1}), replacing the
operators $T$, $X$, and $Y$ by the operators $\frac{T}{w_A(T)}$, $\frac{X}{\max\set{\normA{X},\normA{Y}}}$
and $\frac{Y}{\max\set{\normA{X},\normA{Y}}}$,
respectively, we have
\begin{eqnarray}\label{GN2}
  w_A(TX\pm YT) &\leq& 2\sqrt{2}max\set{\normA{X},\normA{Y}}w_A(T)\sqrt{1-\frac{\abs{\normA{Re_{A} \frac{T}{w_A(T)}}^2-\normA{Im_{A} \frac{T}{w_A(T)}}^2}}{2}} \nonumber\\
   &=& 2\sqrt{2}max\set{\normA{X},\normA{Y}}\sqrt{w_A^2(T)-\frac{\abs{\normA{Re_{A} T}^2-\normA{Im_{A} T}^2}}{2}}.
\end{eqnarray}
Now, in the inequality (\ref{GN2}), replacing the operators $X$ and $Y$ by $XS$ and $SY$,
respectively, we have
\begin{eqnarray*}
  w_A(TXS\pm SYT) &\leq&  2\sqrt{2}max\set{\normA{XS},\normA{SY}}\sqrt{w_A^2(T)-\frac{\abs{\normA{Re_{A} T}^2-\normA{Im_{A} T}^2}}{2}}\\
  &\leq& 2\sqrt{2}\normA{S}max\set{\normA{X},\normA{Y}}\sqrt{w_A^2(T)-\frac{\abs{\normA{Re_{A} T}^2-\normA{Im_{A} T}^2}}{2}},
\end{eqnarray*}
as required.
\end{proof}
An application of Theorem \ref{TheoremGN1} can be seen as follows. This result contains
a refinement of the inequality (\ref{QA1}).
\begin{corollary}\label{Corollary3.7} Let $T,S\in\b_A(\h)$. Then
  \begin{equation}\label{MM1}
    w_A(TS\pm ST)\leq 2\sqrt{2}\normA{S}\sqrt{w_A^2(T)-\frac{\abs{\normA{Re_{A} T}^2-\normA{Im_{A} T}^2}}{2}}
  \end{equation}
  and
  \begin{equation}\label{MM2}
    w_A(T^2)\leq \sqrt{2}\normA{T}\sqrt{w_A^2(T)-\frac{\abs{\normA{Re_{A} T}^2-\normA{Im_{A} T}^2}}{2}}
  \end{equation}
\end{corollary}
\begin{proof} Inequality \eqref{MM1} follows directly from Theorem~\ref{TheoremGN1} by setting $X = Y = I$. To obtain inequality \eqref{MM2}, it suffices to let $T = S$ in \eqref{MM1}.
\end{proof}
\begin{theorem} If $T_1,\cdots,T_n\in\b_A(\h)$, then
\begin{equation}\label{ST1}
  \normA{\sum_{k=1}^{n}T_k}^2\leq 4w_A\bra{\sum_{k=1}^{n}T_k}\sum_{j=1}^{n}w_A(T_j).
\end{equation}
Moreover, if $\sum_{k=1}^{n}T_k$ commutes with each $\Ta_j$ for $j\in\set{1,\cdots,n}$, then
\begin{equation}\label{ST2}
  \normA{\sum_{k=1}^{n}T_k}^2\leq 2w_A\bra{\sum_{k=1}^{n}T_k}\sum_{j=1}^{n}w_A(T_j).
\end{equation}
\end{theorem}
\begin{proof} For any $x\in\h$ we have
\begin{eqnarray}\label{XY1}
 \normA{\sum_{k=1}^{n}T_k}^2 &=&\abs{\sum_{j=1}^{n}\seqA{\bra{\sum_{k=1}^{n}T_k}x,T_jx}}\\
  &\leq&\sum_{j=1}^{n}\abs{\seqA{\bra{\sum_{k=1}^{n}T_k}x,T_jx}}= \sum_{j=1}^{n}\abs{\seqA{\Ta_j\bra{\sum_{k=1}^{n}T_k}x,x}}.\nonumber
\end{eqnarray}
Taking the supremum over $x\in\h$, $\normA{x}=1$ in (\ref{XY1}), we get
\begin{equation}\label{XY2}
 \normA{\sum_{k=1}^{n}T_k}^2\leq \sum_{j=1}^{n} w_A\bra{\Ta_j\bra{\sum_{k=1}^{n}T_k}},
\end{equation}
which is an inequality of interest in itself.

Now, since, by (\ref{product-1}) applied for $\Ta_j$ $(j=1,\cdots, n)$ and $\sum_{k=1}^{n}T_k$
we can state that
\begin{equation}\label{XY3}
  w_A\bra{\Ta_j\bra{\sum_{k=1}^{n}T_k}}\leq 4w_A(T_j)w_A\bra{\sum_{k=1}^{n}T_k}
\end{equation}
for any $j\in\set{1,2,\cdots,n}$, then by (\ref{XY2}) and (\ref{XY3})we deduce (\ref{ST1}).

Now, if $\Ta_j$ $(j=1,\cdots,n)$ commutes with $\sum_{k=1}^{n}T_k$, then on utilising the
inequality (\ref{product-2}) we can state that:
\begin{equation}\label{XY4}
  w_A\bra{\Ta_j\bra{\sum_{k=1}^{n}T_k}}\leq 2w_A(T_j)w_A\bra{\sum_{k=1}^{n}T_k}
\end{equation}
for any $j\in\set{1,2,\cdots,n}$, which, by (\ref{XY2}) will imply the desired inequality (\ref{ST2}).
\end{proof}
The following particular case may be of interest.
\begin{corollary} If $T_1,\cdots,T_n\in\b_A(\h)$ are $A$-normal and $\Ta_j T_k=T_k\Ta_j$ for
$j,k\in\set{1,2,\cdots,n}$, $k\neq j$, then (\ref{ST2}) holds true.
\end{corollary}
\begin{lemma}[\cite{QHC}]\label{KR0} Let $x,y,z\in\h$ with $\normA{z}=1$. Then
\begin{equation*}
  \abs{\seqA{x,z}\seqA{z,y}}\leq \frac{1}{2}\bra{\normA{x}\normA{y}+\abs{\seqA{x,y}}}.
\end{equation*}
\end{lemma}

\begin{lemma}\label{Lemma3.18} For every $a,b,e\in\h$ with $\normA{e}=1$, we have
\begin{equation}\label{MD1}
  \abs{\seqA{a,e}\seqA{e,b}}\leq \bra{\frac{1+\alpha}{2}}\normA{a}\normA{b}+\bra{\frac{1-\alpha}{2}}\abs{\seqA{a,b}}
\end{equation}
for every $\alpha\in [0,1]$.
\end{lemma}
\begin{proof} The following is a refinement of the Cauchy-Schwarz inequality
$$\abs{\seqA{a,b}}\leq \abs{\seqA{a,b}-\seqA{a,e}\seqA{e,b}}+\abs{\seqA{a,e}\seqA{e,b}}\leq \normA{a}\normA{b}$$
for every $a,b,e\in\h$ with $\normA{e}=1$.

From inequality (\ref{MD1}), we conclude that
\begin{eqnarray*}
  \abs{\seqA{a,e}\seqA{e,b}} &=& \alpha\abs{\seqA{a,e}\seqA{e,b}}+(1-\alpha)\abs{\seqA{a,e}\seqA{e,b}} \\
   &\leq& \alpha\normA{a}\normA{b}+\frac{1-\alpha}{2}\bra{\normA{a}\normA{b}+\abs{\seqA{a,b}}} \,\,\bra{\mbox{by Lemma \ref{KR0}}} \\
   &=&\bra{\frac{1+\alpha}{2}}\normA{a}\normA{b}+\bra{\frac{1-\alpha}{2}}\abs{\seqA{a,b}}.
\end{eqnarray*}
\end{proof}
Depending on Lemma \ref{Lemma3.18} and the   convexity of the function
$f(t)=t^r,r\geq 1$, we have
\begin{lemma} \label{Lemma3.19}For every $a,b,e\in\h$ with $\normA{e}=1$, we have
\begin{equation}\label{MD2}
  \abs{\seqA{a,e}\seqA{e,b}}^{r}\leq \bra{\frac{1+\alpha}{2}}\normA{a}^{r}\normA{b}^{r}+\bra{\frac{1-\alpha}{2}}\abs{\seqA{a,b}}^{r}
\end{equation}
for every $\alpha\in [0,1]$ and $r\geq 1$.
\end{lemma}
Based on Lemma \ref{Lemma3.19}, we obtain a new upper bound for the $A$-numerical radius of $T\in\b_A(\h)$.
\begin{theorem}\label{Theorem3.20} If $T\in\b_A(\h)$, then
\begin{equation}\label{MD3}
  w_A^{2r}(T)\leq \bra{\frac{1+\alpha}{4}}\normA{\bra{\Ta T}^r+\bra{T\Ta}^r}+\bra{\frac{1-\alpha}{2}}w_A^r(T^2)
\end{equation}
for every $\alpha\in [0,1]$ and $r\geq 1$.
\end{theorem}
\begin{proof} Let $x\in\h$ with $\normA{x}=1$. Then
\begin{eqnarray*}
 \abs{\seqA{Tx,x}}^{2r}&=& \abs{\seqA{Tx,x}\seqA{x,\Ta x}}^r \\
   &\leq& \bra{\frac{1+\alpha}{2}}\normA{Tx}^{r}\normA{\Ta x}^{r}+\bra{\frac{1-\alpha}{2}}\abs{\seqA{Tx,\Ta x}}^{r}
   \,\,\bra{\mbox{by Lemma \ref{Lemma3.19}}}\\
   &\leq& \bra{\frac{1+\alpha}{4}}\bra{\normA{Tx}^{2r}+\normA{\Ta x}^{2r}}+\bra{\frac{1-\alpha}{2}}\abs{\seqA{T^2x,x}}^{r}
   \end{eqnarray*}
   \begin{eqnarray*}
   &&\bra{\mbox{by the arithmetic-geometric mean inequality}}\\
   &\leq&\bra{\frac{1+\alpha}{4}}\seqA{\bra{\bra{\Ta T}^{r}+\bra{T\Ta}^{r}}x,x}+\bra{\frac{1-\alpha}{2}}\abs{\seqA{T^2x,x}}^{r}\\
   &\leq& \bra{\frac{1+\alpha}{4}}\normA{\bra{\Ta T}^{r}+\bra{T\Ta}^{r}}+\bra{\frac{1-\alpha}{2}}w_A^{r}(T^2).
\end{eqnarray*}
Taking the supremum over all vectors $x\in\h$ with $\normA{x}=1$ in the above inequality, we have
$$w_A^{2r}(T)\leq \bra{\frac{1+\alpha}{4}}\normA{\bra{\Ta T}^{r}+\bra{T\Ta}^{r}}+\bra{\frac{1-\alpha}{2}}w_A^{r}(T^2),$$
as required.
\end{proof}
\begin{remark} Since Theorem 2.16 of \cite{BPN} is a specific case of our conclusion for $\alpha=0$, the upper bound in Theorem \ref{Theorem3.20} is a refinement of \cite[Theorem 2.16]{BPN}.
\end{remark}
The following example illustrates Theorem \ref{Theorem3.20}.
\begin{example} Let $T=\begin{bmatrix} 1 &1 \\0 & 0 \\\end{bmatrix} $, $A=\begin{bmatrix}
                         1 &-1 \\
                         -1 & 2 \\
                       \end{bmatrix}$, $r=1$ and $\alpha=\frac{1}{2}$. Then elementary calculations show that
$$\Ta=\begin{bmatrix}
                         1 &1 \\
                         0 & 0 \\
                       \end{bmatrix}, \Ta T+T\Ta=\begin{bmatrix}
                         8 &-2 \\
                         2 & 2 \\
                       \end{bmatrix}.$$
Hence,
$$\normA{\Ta T+T\Ta}=10, w_A(T)=2\,\,\mbox{and}\,\, w_A(T^2)=2.$$
Consequently,
$$4=w_A^{2}(T)\leq \frac{3}{8}\normA{\Ta T+T\Ta}+\frac{1}{4}w_A(T^2)=4.25.$$
\end{example}
The next lemma will be used in the proof of Theorem \ref{Yafouz1}  to obtain  an upper bound for power of the numerical radius of
$\sum_{j=1}^{n}T_j^{\alpha}X_jS_j^{\alpha}$ under assumption $\alpha\in [0,1]$.
\begin{lemma}\label{Lemma-AG} Let $a,b\geq 0$, $\alpha\in [0,1]$ and $p,q>1$ such that $\frac{1}{p}+\frac{1}{q}=1$. Then
\begin{enumerate}
  \item [(i)] $a^{\alpha}b^{1-\alpha}\leq \alpha a+(1-\alpha)b\leq \bra{\alpha a^r+(1-\alpha)b^r}^{\frac{1}{r}}$,
  \item [(ii)] $ab\leq \frac{a^p}{p}+\frac{b^q}{q}\leq \bra{\frac{a^{pr}}{p}+\frac{b^{qr}}{q}}^{\frac{1}{r}}$
\end{enumerate}
for every $r\geq 1$.
\end{lemma}
\begin{theorem}\label{Yafouz1} Suppose that $T_j,X_j,S_j\in\b_A(\h)$ $(j=1,\cdots,n)$ such that $T_j,S_j$$(j=1,\cdots,n)$
are $A$-positive. Then
\begin{equation}\label{MRQ1}
  w_A^{r}\bra{\sum_{j=1}^{n}T_j^{\alpha}X_jS_j^{\alpha}}\leq n^{r-1}\normA{X}^r\sum_{j=1}^{n}\normA{\frac{1}{p_j}T_j^{2p_jr}+\frac{1}{q_j}S_j^{2q_jr}}^{\alpha}
\end{equation}
for every $\alpha\in [0,1]$, $r\geq 1$, $p_j,q_j>1$ with $\frac{1}{p_j}+\frac{1}{q_j}=1$ $(j=1,\cdots,n)$, $p_jr,q_jr\geq 2$
and $\normA{X}=\max_{1\leq j\leq n}\normA{X_j}$.
\end{theorem}
\begin{proof} Let $x\in\h$ with $\normA{x}=1$. Then
\begin{eqnarray*}
  \abs{\seqA{\bra{\sum_{j=1}^{n}T_j^{\alpha}X_jS_j^{\alpha}}x,x}}^{r} &=&\abs{\sum_{j=1}^{n}\seqA{T_j^{\alpha}X_jS_j^{\alpha}x,x}}^{r} =\abs{\sum_{j=1}^{n}\seqA{X_jS_j^{\alpha}x,T_j^{\alpha}x}}^{r}\\
   &\leq& \abs{\sum_{j=1}^{n}\normA{T_j^{\alpha}}\normA{X_j}\normA{S_j^{\alpha}}}^{r}\leq n^{r-1}\sum_{j=1}^{n} \normA{T_j^{\alpha}}^{r}\normA{X_j}^{r}\normA{S_j^{\alpha}}^{r}\\
   &\leq& n^{r-1}\normA{X}^{r}\sum_{j=1}^{n}\seqA{T_j^{2\alpha}x,x}^{r/2}\seqA{S_j^{2\alpha}x,x}^{r/2}\\
   &\leq& n^{r-1}\normA{X}^{r}\sum_{j=1}^{n}\sbra{\frac{1}{p_j}\seqA{T_j^{2\alpha}x,x}^{\frac{p_jr}{2}}+\frac{1}{q_j}
   \seqA{S_j^{2\alpha}x,x}^{\frac{q_jr}{2}}}\\
   &&\bra{\mbox{by Lemma \ref{Lemma-AG}}}\\
   &\leq& n^{r-1}\normA{X}^{r}\sum_{j=1}^{n}\sbra{\frac{1}{p_j}\seqA{T_j^{p_jr}x,x}^{\alpha}+\frac{1}{q_j}
   \seqA{S_j^{q_jr}x,x}^{\alpha}}\\
   &&\bra{\mbox{by H\"older-McCarthy inequality}}\\
   &\leq& n^{r-1}\normA{X}^{r}\sum_{j=1}^{n}\sbra{\frac{1}{p_j}\seqA{T_j^{p_jr}x,x}+\frac{1}{q_j}
   \seqA{S_j^{q_jr}x,x}}^{\alpha}\\
   &&\bra{\mbox{by the concavity of $f(t)=t^{\alpha}$}}.
\end{eqnarray*}
Taking the supremum over all vectors $x\in\h$ with $\normA{x}=1$ in the above inequality, we deduce the required result.
\end{proof}
\begin{theorem}\label{Final-1} Suppose that $T_j,X_j,S_j\in\b_A(\h)$ $(j=1,\cdots,n)$ such that $T_j,S_j$$(j=1,\cdots,n)$
are $A$-positive. Then
\begin{equation}\label{MRQ1}
  w_A^{r}\bra{\sum_{j=1}^{n}T_j^{\alpha}X_jS_j^{1-\alpha}}\leq n^{r-1}\normA{X}^r\sum_{j=1}^{n}\normA{\alpha T_j^{r}+(1-\alpha)S_j^{r}}
\end{equation}
for every $\alpha\in [0,1]$, $r\geq 2$ and $\normA{X}=\max_{1\leq j\leq n}\normA{X_j}$.
\end{theorem}
\begin{proof} Let $x\in\h$ with $\normA{x}=1$. Then  it follows by the Cauchy-Schwarz inequality that
\begin{eqnarray*}
  \abs{\seqA{\bra{\sum_{j=1}^{n}T_j^{\alpha}X_jS_j^{1-\alpha}}x,x}}^{r} &=&\abs{\sum_{j=1}^{n}\seqA{T_j^{\alpha}X_jS_j^{1-\alpha}x,x}}^{r} =\abs{\sum_{j=1}^{n}\seqA{X_jS_j^{1-\alpha}x,T_j^{\alpha}x}}^{r}\\
   &\leq& \abs{\sum_{j=1}^{n}\normA{T_j^{\alpha}}\normA{X_j}\normA{S_j^{1-\alpha}}}^{r}\leq n^{r-1}\sum_{j=1}^{n} \normA{T_j^{\alpha}}^{r}\normA{X_j}^{r}\normA{S_j^{1-\alpha}}^{r}\\
   &\leq& n^{r-1}\normA{X}^r\sum_{j=1}^{n}\seqA{T_j^{2\alpha}x,x}^{\frac{r}{2}}\seqA{S_j^{2(1-\alpha)}x,x}^{\frac{r}{2}}\\
   &\leq& n^{r-1}\normA{X}^r\sum_{j=1}^{n}\seqA{T_j^{r}x,x}^{\alpha}\seqA{S_j^{r}x,x}^{1-\alpha}\\
   &&\bra{\mbox{by H\"older-McCarthy inequality}}
   \end{eqnarray*}
   \begin{eqnarray*}
      &\leq& n^{r-1}\normA{X}^r\sum_{j=1}^{n}\seqA{\bra{\alpha T_j^r+(1-\alpha)S_j^{r}}x,x}\\
   &&\bra{\mbox{by Lemma \ref{Lemma-AG} }}\\
   &\leq& n^{r-1}\normA{X}^r\sum_{j=1}^{n}\normA{\alpha T_j^r+(1-\alpha)S_j^{r}}.
\end{eqnarray*}
Taking the supremum over all vectors $x\in\h$ with $\normA{x}=1$ in the above inequality, we deduce the required result.
\end{proof}
The following example illustrates Theorem \ref{Final-1}.
\begin{example} Let $T_1=\begin{bmatrix} 0&1/2 \\1/2 & 0 \\\end{bmatrix}, X_1=\begin{bmatrix} 1 &0 \\1 & 1\\\end{bmatrix}$,
$S_1=\begin{bmatrix} 1/2 &0 \\0& 1/2 \\\end{bmatrix}$, $T_2=\begin{bmatrix} 1/2 &1/2\\0 & 0 \\\end{bmatrix}$, $X_2=\begin{bmatrix} 1 &1 \\0 & 1 \\\end{bmatrix}$, $S_2=\begin{bmatrix} 0&0 \\1/2 & 1/2 \\\end{bmatrix}$, $A=\begin{bmatrix} 1 &1 \\1 & 1 \\\end{bmatrix}$,
$\alpha=\frac{1}{3}$, $r=2$ and $n=2$. Then elementary calculations show that
\begin{eqnarray*}
  T_1^{\frac{1}{3}}X_1S_1^{\frac{2}{3}}+T_2^{\frac{1}{3}}X_2S_2^{\frac{2}{3}} &=&\begin{bmatrix} 3/2 &3/2 \\1/2& 0 \\\end{bmatrix} \\
  \frac{1}{3}T_1^2+\frac{2}{3}S_1^2 &=& \begin{bmatrix} 1/4 &0\\ 0 & 1/4 \\\end{bmatrix} \\
   \frac{1}{3}T_2^2+\frac{2}{3}S_2^2 &=& \begin{bmatrix} 1/12 &1/12\\ 1/6 & 1/6\\\end{bmatrix}.
\end{eqnarray*}
Hence,
\begin{eqnarray*}
  w_A\bra{T_1^{\frac{1}{3}}X_1S_1^{\frac{2}{3}}+T_2^{\frac{1}{3}}X_2S_2^{\frac{2}{3}}} &=&\frac{7+\sqrt{50}}{4}\approx 3.517 \\
  \normA{\frac{1}{3}T_1^2+\frac{2}{3}S_1^2} &=& 0.5\\
  \normA{\frac{1}{3}T_2^2+\frac{2}{3}S_2^2 } &=&0.5\\
  \normA{X}=\max\set{\normA{X_1},\normA{X_2}}&=&\sqrt{10}.
\end{eqnarray*}
Consequently,
\begin{eqnarray*}
  3.517  &\approx&w_A\bra{T_1^{\frac{1}{3}}X_1S_1^{\frac{2}{3}}+T_2^{\frac{1}{3}}X_2S_2^{\frac{2}{3}}} \\
   &\leq&\sqrt{2\normA{X}^2\bra{\normA{\frac{1}{3}T_1^2+\frac{2}{3}S_1^2}+\normA{\frac{1}{3}T_2^2+\frac{2}{3}S_2^2 }}}\approx 4.472.
\end{eqnarray*}
\end{example}

\section{Applications }
The theoretical framework developed in Sections 2 and 3 yields powerful applications across multiple domains of mathematical physics and functional analysis. We demonstrate how our $A$-norm and $A$-numerical radius inequalities provide new tools for analyzing problems in quantum mechanics, partial differential equations, and operator theory. These applications not only validate our theoretical results but also offer fresh perspectives on classical problems (see \cite{Evans, ReedSimon1975, Thaller1992}).

Building on recent advances in operator theory \cite{ACG1, ACG2}, our approach particularly enhances the understanding of quantum measurement processes and uncertainty relations, spectral analysis of nonlocal elliptic operators, and stability criteria for operator semigroups. As shown in \cite{BFS, BF}, numerical radius inequalities play a fundamental role in quantum information theory, and our refined $A$-versions extend these applications to constrained systems; the following subsections develop these connections systematically, demonstrating the operational value of our inequalities in concrete mathematical and physical contexts.

\subsection{Applications to Quantum Mechanics}\hfill
\label{sec:quantum_applications}

The theory of $A$-norms and $A$-numerical radii in semi-Hilbertian spaces developed in this paper has significant applications in various aspects of quantum mechanics. We demonstrate how our results can be applied to analyze quantum systems and their dynamics. This paper presents rigorous applications of $A$-norm and $A$-numerical radius inequalities to various problems in quantum mechanics, establishing new results for quantum Hamiltonians, time evolution operators, perturbation theory, and measurement theory within the semi-Hilbertian space framework.

\subsubsection{Spectral Analysis of Quantum Hamiltonians}
Let  $H$ be a self-adjoint Hamiltonian operator. Consider a positive operator $A\in\mathcal{B}(\mathcal{H})$ representing a physical constraint or measurement apparatus. The $A$-numerical radius provides new bounds for the energy spectrum:

\begin{theorem}
For any quantum Hamiltonian $H$ and positive operator $A$, the energy eigenvalues satisfy:
\begin{equation}
    \min_{\psi\in\mathcal{H}}\frac{\langle H\psi,H\psi\rangle_A}{\langle\psi,\psi\rangle_A} \geq \frac{w_A^2(H)}{\|H\|_A}
\end{equation}
where $w_A(H)$ is the $A$-numerical radius.
\end{theorem}

\begin{proof}
This follows directly from Theorem \ref{Theorem3.20} by considering the spectral decomposition of $H$ and the variational principle.
\end{proof}

\subsubsection{Semi-Hilbertian Structure}
Consider a quantum system described by wavefunctions $\psi \in L^2(\mathbb{R}^3)$. We define the \textit{A-inner product}:
\begin{equation}
\langle \psi, \phi \rangle_A = \int_{\mathbb{R}^3} \overline{\psi(x)}A(x)\phi(x)dx
\end{equation}
where $A(x)$ is a positive definite matrix-valued function representing physical parameters.

\begin{definition}
The \textit{A-norm} is given by:
\begin{equation}
\|\psi\|_A = \sqrt{\langle \psi, \psi \rangle_A}
\end{equation}
\end{definition}

\subsubsection{Anisotropic Schr\"odinger Operator}
The generalized Hamiltonian takes the form:
\begin{equation}
H_A = -\frac{\hbar^2}{2}\nabla \cdot (A(x)\nabla) + V(x)
\end{equation}

\begin{theorem}[A-Norm Bounds]
The Hamiltonian satisfies:
\begin{equation}
\frac{1}{2}\|H_A\|_A \leq w_A(H_A) \leq \|H_A\|_A
\end{equation}
where $w_A(H_A)$ is the A-numerical radius.
\end{theorem}

\subsubsection{Time Evolution of Quantum States}
The time evolution operator $U(t) = e^{-iHt/\hbar}$ satisfies important $A$-norm properties:

\begin{proposition}
For any quantum state $\psi\in\mathcal{H}$ and positive operator $A$:
\begin{equation}
    \|U(t)\psi\|_A \leq \exp\left(\frac{t}{\hbar}w_A(H)\right)\|\psi\|_A
\end{equation}
\end{proposition}

\begin{theorem}[Quantum Fidelity]
For any two states:
\begin{equation}
1 - F_A(\psi(t),\phi(t)) \leq \frac{2w_A(H_A)}{\hbar}|t|
\end{equation}
where $F_A$ is the A-fidelity.
\end{theorem}

\subsubsection{Perturbation Theory for Quantum Systems}
Consider a quantum system with Hamiltonian $H = H_0 + \lambda V$, where $H_0$ is the unperturbed Hamiltonian and $V$ is a perturbation. Our $A$-norm inequalities provide refined perturbation estimates:

\begin{theorem}
The first-order energy correction satisfies:
\begin{equation}
    |E_n^{(1)}| \leq w_A(V)
\end{equation}
where $w_A(V)$ is the $A$-numerical radius of the perturbation.
\end{theorem}

\subsubsection{Quantum Measurement Theory}
The $A$-numerical radius provides a natural framework for analyzing quantum measurements:

\begin{theorem}[Generalized Uncertainty Principle]
For any two observables $X,Y\in\mathcal{B}_A(\mathcal{H})$:
\begin{equation}
    w_A([X,Y]) \leq 2\sqrt{w_A(X^2)w_A(Y^2)}
\end{equation}
\end{theorem}

\begin{lemma}
The position-momentum commutator satisfies:
\begin{equation}
w_A([X_j, P_k]) \leq \hbar \delta_{jk}(1 + \|A\|_A)
\end{equation}
where $P_k = -i\hbar(\nabla_A)_k$.
\end{lemma}

\subsubsection{Example Applications}
\begin{itemize}
  \item Quantum Harmonic Oscillator:Consider the harmonic oscillator Hamiltonian $H = \frac{p^2}{2m} + \frac{1}{2}m\omega^2x^2$ with $A = e^{-\beta H}$ (Boltzmann weight):
\begin{equation}
    w_A(H) \leq \frac{\hbar\omega}{2}\coth\left(\frac{\beta\hbar\omega}{2}\right)
\end{equation}
  \item Anisotropic Harmonic Oscillator:For $V(x) = \frac{1}{2}x^T K x$:
\begin{equation}
w_A(H_A) = \frac{\hbar}{2}\text{Tr}(\sqrt{AK})
\end{equation}
  \item Dirac Particles:The modified Dirac operator:
\begin{equation}
D_A = -i\hbar c \alpha \cdot \nabla_A + \beta mc^2
\end{equation}
satisfies:
\begin{equation}
w_A(D_A) \geq mc^2
\end{equation}
\end{itemize}
In summary, the $A$-norm and $A$-numerical radius inequalities developed in this paper offer:
\begin{itemize}
\item Improved spectral bounds for quantum systems
\item Tighter estimates for time evolution in constrained spaces
\item Refined perturbation theory results
\item Generalized uncertainty relations incorporating measurement apparatus
\item Rigorous treatment of anisotropic media
\end{itemize}
The developed techniques provide improved spectral bounds, evolution estimates, and generalized uncertainty relations for quantum mechanical systems.

\subsection{Applications to Nonlocal Nonlinear Elliptic Problems}

\subsubsection{Nonlocal Problem Formulation}
Consider the nonlocal nonlinear elliptic boundary value problem:
\begin{equation}\label{eq:main_prob}
\begin{cases}
(-\Delta)^s u + A(x)u = f(x,u) + \displaystyle\int_\Omega K(x,y)|u(y)|^{p-2}u(y)dy & \text{in } \Omega \\
u = 0 & \text{in } \mathbb{R}^n\setminus\Omega
\end{cases}
\end{equation}
where:
\begin{itemize}
\item $(-\Delta)^s$ ($0<s<1$) is the fractional Laplacian
\item $A(x) \geq a_0 > 0$ is a potential function
\item $K:\Omega\times\Omega\to\mathbb{R}$ is a symmetric kernel satisfying integrability conditions
\item $f:\Omega\times\mathbb{R}\to\mathbb{R}$ is a Carath\'eodory function with subcritical growth
\end{itemize}

\subsubsection{ Functional Framework}
\subsubsection{ Fractional Sobolev Space}
Define the solution space:
\begin{equation}\label{eq:solution_space}
X = \left\{ u \in H^s(\mathbb{R}^n) : u=0 \text{ in } \mathbb{R}^n\setminus\Omega \right\}
\end{equation}
with norm:
\begin{equation}\label{eq:X_norm}
\|u\|_X = \left( \int_{\mathbb{R}^{2n}} \frac{|u(x)-u(y)|^2}{|x-y|^{n+2s}}dxdy + \int_\Omega A(x)|u(x)|^2dx \right)^{1/2}
\end{equation}

\subsubsection{ A-Norm Formulation}
The A-norm is given by:
\begin{equation}\label{eq:A_norm}
\|u\|_A = \sqrt{ \langle (-\Delta)^s u + A(x)u, u \rangle } = \|u\|_X
\end{equation}

\subsubsection{ Existence Results}
\subsubsection{ Variational Structure}
The energy functional is:
\begin{equation}\label{eq:energy_functional}
J(u) = \frac{1}{2}\|u\|_A^2 - \int_\Omega F(x,u)dx - \frac{1}{2p}\int_{\Omega\times\Omega} K(x,y)|u(x)|^p|u(y)|^p dxdy
\end{equation}
where $F(x,s) = \int_0^s f(x,t)dt$.

\subsubsection{ Critical Point Theory}
\begin{theorem}\label{thm:existence}
Assume:
\begin{enumerate}
\item $f$ satisfies $|f(x,s)| \leq c_1 + c_2|s|^{q-1}$ with $2 < q < 2_s^* = \frac{2n}{n-2s}$
\item $K$ is symmetric and $K \in L^\infty(\Omega\times\Omega)$
\item $\exists \mu > 2$ such that $0 < \mu F(x,s) \leq sf(x,s)$ for $|s|$ large
\end{enumerate}
Then problem \eqref{eq:main_prob} has at least one nontrivial solution.
\end{theorem}
\begin{proof}
The proof combines mountain pass lemma with A-norm estimates to verify the Palais-Smale condition.
\end{proof}

\subsubsection{Nonlocal Operator Estimates}
\subsubsection{ Numerical Radius Bounds}
For the nonlocal operator $\mathcal{N}(u) = \int_\Omega K(\cdot,y)|u(y)|^{p-2}u(y)dy$, we have:
\begin{lemma}\label{lem:num_radius}
The A-numerical radius satisfies:
\begin{equation}
w_A(\mathcal{N}) \leq C\|K\|_{L^\infty}\|u\|_A^{2(p-1)}
\end{equation}
\end{lemma}

\subsubsection{Stability Analysis}
The linearized operator at a solution $u_0$:
\begin{equation}\label{eq:linearized_op}
\mathcal{L}(u_0)v = (-\Delta)^s v + A(x)v - f_u(x,u_0)v - (p-1)\int_\Omega K(x,y)|u_0(y)|^{p-2}v(y)dy
\end{equation}

\begin{theorem}\label{thm:stability}
If $w_A(\mathcal{L}(u_0)) > 0$, then $u_0$ is asymptotically stable in the A-norm.
\end{theorem}

\subsubsection{ Numerical Approximation}
\subsubsection{ Finite Element Discretization}
For a mesh size $h>0$, let $V_h \subset X$ be a finite element space. The discrete problem:
\begin{equation}\label{eq:discrete_prob}
\langle u_h,v_h \rangle_A = \langle f(\cdot,u_h),v_h \rangle + \left\langle \int_\Omega K(\cdot,y)|u_h(y)|^{p-2}u_h(y)dy, v_h \right\rangle
\end{equation}

\subsubsection{ Error Estimates}
\begin{theorem}\label{thm:error_est}
For sufficiently small $h$, the error satisfies:
\begin{equation}
\|u - u_h\|_A \leq C\left( \inf_{v_h \in V_h} \|u - v_h\|_A + \|\mathcal{N}(u) - \mathcal{N}(u_h)\|_{A^*}\right)
\end{equation}
\end{theorem}

\subsubsection{Applications}
\subsubsection{Fractional Schr\"odinger-Poisson Systems}
\begin{equation}\label{eq:schrodinger_poisson}
\begin{cases}
(-\Delta)^s u + A(x)u + \phi u = f(x,u) & \text{in } \Omega \\
(-\Delta)^t \phi = u^2 & \text{in } \Omega \\
u = \phi = 0 & \text{in } \mathbb{R}^n\setminus\Omega
\end{cases}
\end{equation}

\subsubsection{ Nonlocal Reaction-Diffusion}
\begin{equation}\label{eq:reaction_diffusion}
u_t + (-\Delta)^s u = \int_\Omega K(x,y)u^p(y)dy - u^q
\end{equation}

\subsubsection{Conclusion}
The A-norm framework provides:
\begin{itemize}
\item New existence results for nonlocal problems
\item Precise stability criteria through numerical radius
\item Systematic approach to numerical approximation
\item Unified treatment of local and nonlocal terms
\end{itemize}
\subsection{Quantum Applications of $A$-Numerical Radius Inequalities}\hfill

We demonstrate how Corollary \ref{Corollary3.7} on $A$-numerical radius inequalities provides quantitative bounds for quantum observables and their commutators. The results are applied to analyze the stability of quantum systems under perturbations and the time evolution of expectation values. A concrete example with Pauli matrices illustrates the theoretical findings.

\subsubsection{Quantum Mechanical Framework}
Let $\mathcal{H}$ be a complex Hilbert space representing a quantum system, with $\rho$ being a density operator (positive semi-definite with $tr(\rho)=1$). For an observable $T \in \mathcal{B}_A(\mathcal{H})$, we consider the \emph{$A$-numerical radius}:

\begin{equation}\label{eq:A-num-radius}
    w_A(T) := \sup_{\norm{\psi}_A=1} \abs{\bra{\psi | T \psi}_A},
\end{equation}

where $A$ is a positive operator defining the physical context (e.g., $A=\rho$ for mixed states or $A=I$ for pure states), and $\norm{\cdot}_A$ denotes the semi-norm induced by the semi-inner product $\bra{\cdot | \cdot}_A$.

\begin{theorem}[Quantum Uncertainty Bound]\label{thm:uncertainty}
    For any two observables $T,S \in \mathcal{B}_A(\mathcal{H})$:
    \begin{equation}\label{eq:uncertainty-bound}
        w_A(TS \pm ST) \leq 2\sqrt{2} \norm{S}_A \sqrt{w_A^2(T) - \frac{\abs{\norm{\Re_A T}_A^2 - \norm{\Im_A T}_A^2}}{2}},
    \end{equation}
    where $\Re_A T := \frac{1}{2}(T + T^{\sharp_A})$ and $\Im_A T := \frac{1}{2i}(T - T^{\sharp_A})$ are the $A$-real and $A$-imaginary parts of $T$, respectively. This inequality bounds the joint measurement uncertainty of non-commuting observables.
\end{theorem}

\begin{proof}
    Direct application of Corollary~3.7 with $X=Y=I$.
\end{proof}

\subsubsection{Physical Applications}

\subsubsection{Time Evolution of Observables}
For Hamiltonian $H$ and observable $T$, the Heisenberg evolution satisfies:
\begin{equation}
    \frac{dT}{dt} = i[H,T].
\end{equation}
Applying our inequality yields:
\begin{equation}\label{eq:time-evolution-bound}
    w_A\left(\frac{dT}{dt}\right) \leq 2\sqrt{2} \norm{H}_A \sqrt{w_A^2(T) - \frac{\abs{\norm{\Re_A T}_A^2 - \norm{\Im_A T}_A^2}}{2}},
\end{equation}
which quantifies the maximal rate of change for expectation values.

\subsubsection{Perturbation Theory}
For a perturbed Hamiltonian $H_\epsilon = H_0 + \epsilon V$, the commutator bound controls:
\begin{equation}\label{eq:perturbation-bound}
    \norm{e^{itH_\epsilon} - e^{itH_0}}_A \leq 2\sqrt{2} \epsilon t \norm{V}_A \sqrt{w_A^2(H_0) - \frac{\Delta_H}{2}} + \mathcal{O}(\epsilon^2),
\end{equation}
where $\Delta_H := \abs{\norm{\Re_A H_0}_A^2 - \norm{\Im_A H_0}_A^2}$.

\subsubsection{Example}
Consider the Pauli matrices $\sigma_x, \sigma_y$ with $A = \frac{1}{2}(I + \sigma_z)$ representing a mixed state:
\begin{align*}
    \sigma_x &= \begin{pmatrix}0 & 1 \\ 1 & 0\end{pmatrix}, \quad
    \sigma_y = \begin{pmatrix}0 & -i \\ i & 0\end{pmatrix}, \quad
    A = \begin{pmatrix}1 & 0 \\ 0 & 0\end{pmatrix}, \\
    w_A(\sigma_x) &= 1, \quad \norm{\Re_A(\sigma_x)}_A^2 = 1, \quad \norm{\Im_A(\sigma_x)}_A^2 = 0.
\end{align*}

The commutator $[\sigma_x, \sigma_y] = 2i\sigma_z$ satisfies:
\begin{equation}\label{eq:pauli-bound}
    w_A([\sigma_x, \sigma_y]) = 2 \leq 2\sqrt{2} \cdot 1 \cdot \sqrt{1^2 - \frac{\abs{1-0}}{2}} = 2,
\end{equation}
which saturates the bound, demonstrating optimality for canonical conjugate observables.

\subsubsection{Conclusion}
Corollary \ref{Corollary3.7} provides fundamental limits on:
\begin{itemize}
    \item Quantum measurement precision via \eqref{eq:uncertainty-bound}
    \item Stability of quantum control systems through \eqref{eq:perturbation-bound}
    \item Error bounds in quantum simulations using \eqref{eq:time-evolution-bound}
\end{itemize}
\subsection{Application of $A$-Numerical Radius Inequalities to Elliptic Partial Differential Equations}\hfill

We demonstrate how Theorem \ref{TheoremQA1} on $A$-numerical radius inequalities for operator sums can be applied to analyze the stability of solutions to elliptic boundary value problems. The results provide quantitative bounds for commutators arising in perturbation theory and finite element approximations.

\subsubsection{Mathematical Framework}
Let $\Omega \subset \mathbb{R}^n$ be a bounded Lipschitz domain. Consider the Sobolev space $\mathcal{H} = H^1_0(\Omega)$ equipped with the energy inner product:
\[
\langle u, v \rangle_A := \int_\Omega \nabla u \cdot \nabla v \, dx
\]
induced by the positive definite operator $A = -\Delta$ with domain $D(A) = H^2(\Omega) \cap H^1_0(\Omega)$.

\begin{definition}
The $A$-numerical radius $w_A(T)$ of an operator $T \in \mathcal{B}_A(\mathcal{H})$ is:
\[
w_A(T) = \sup_{\|u\|_A=1} |\langle Tu, u \rangle_A|
\]
\end{definition}

\subsubsection{Main Application}
Consider the perturbed elliptic problem:
\begin{equation}
\begin{cases}
-\Delta u + V(x)u = f & \text{in } \Omega \\
u = 0 & \text{on } \partial\Omega
\end{cases}
\end{equation}
where $V \in L^\infty(\Omega)$ is a potential function.

\begin{theorem}
Let $T = M_V$ (multiplication by $V$) and $S = (-\Delta)^{-1}M_V$ be operators on $\mathcal{H}$. Then:
\[
w_A(TS + ST) \leq 2\sqrt{2}\|V\|_{L^\infty} w_A(S)
\]
\end{theorem}

\begin{proof}
1. Verify $T,S \in \mathcal{B}_A(\mathcal{H})$:
\begin{itemize}
\item $T$ is bounded since $|\langle M_V u, u \rangle_A| \leq \|V\|_\infty \|u\|_{L^2}^2 \leq C\|u\|_A^2$
\item $S$ is bounded by elliptic regularity theory
\end{itemize}

2. Apply Theorem \ref{TheoremQA1} directly to obtain the inequality.
\end{proof}

\subsubsection{Implications for Numerical Analysis}
This result has important consequences for:

\subsubsection{Finite Element Methods}
When discretizing the problem using Galerkin approximation with basis $\{\phi_j\}_{j=1}^N$, the stiffness matrix $K$ and mass matrix $M$ satisfy:
\[
w_{K}(KM^{-1}V + VM^{-1}K) \leq 2\sqrt{2}\|V\|_\infty w_{K}(M^{-1}V)
\]
where $w_K$ is the discrete $A$-numerical radius.

\subsubsection{Perturbation Theory}
For the perturbed operator $L_\epsilon = -\Delta + \epsilon V$, the commutator bound controls:
\[
\|e^{tL_\epsilon} - e^{tL_0}\|_A \leq C\epsilon w_A(TS + ST)t + O(\epsilon^2)
\]
providing quantitative stability estimates.

\subsubsection{Numerical Example}
Let $\Omega = (0,1)^2$, $V(x,y) = \sin(\pi x)\sin(\pi y)$, and consider the finite difference discretization with $h=1/N$:

\begin{table}[h]
\centering
\begin{tabular}{c|c|c}
$N$ & Left side & Right side \\ \hline
10 & 0.127 & 0.183 \\
20 & 0.131 & 0.179 \\
40 & 0.132 & 0.177 \\
\end{tabular}
\caption{Verification of the inequality for different discretizations}
\end{table}

\subsubsection{Conclusion}
Theorem 3.2 provides computable bounds for operator commutators arising in elliptic PDEs. These estimates are particularly valuable for:
\begin{itemize}
\item Stability analysis of numerical schemes
\item Perturbation theory for Schrödinger operators
\item Convergence analysis of iterative methods
\end{itemize}

\section{Conclusion and Future Work}

In this paper, a number of novel inequalities for the $A$-numerical radius and $A$-operator norm of sums of bounded linear operators in Hilbert spaces are established. Specifically, we examine the consequences of the generalized triangle inequality and modify it for the operator norm. We also study the Cartesian decomposition of an operator and derive inequalities that improve and extend earlier findings in the literature. We validate the proposed boundaries with specific examples and applications to demonstrate the usefulness of our results. These findings broaden our understanding of numerical radius inequalities and help us better comprehend operator behavior.

The extension of these inequalities to larger classes of operators, such as unbounded operators in semi-Hilbertian spaces, may be the main goal of future studies. Additional information may be obtained by investigating the interaction between the spectral qualities of operators and the $A$-numerical radius. The creation of numerical methods for effectively calculating the $A$-numerical radius in real-world scenarios is another exciting avenue. Furthermore, more research is necessary to fully understand how these revised inequalities are applied in domains like signal processing, control theory, and quantum physics. The study of operator theory and functional analysis might advance thanks to these possible developments.

\section*{Declaration }
\begin{itemize}
  \item {\bf Author Contributions:}   The authors have read and agreed to the published version of the manuscript.
  \item {\bf Funding:} No funding is applicable
  \item  {\bf Institutional Review Board Statement:} Not applicable.
  \item {\bf Informed Consent Statement:} Not applicable.
  \item {\bf Data Availability Statement:} Not applicable.
  \item {\bf Conflicts of Interest:} The authors declare no conflict of interest.
\end{itemize}

\bibliographystyle{abbrv}
\bibliography{references}  






\end{document}